\documentclass{gtpart}     
\usepackage[all]{xy}
\usepackage{verbatim}
\title{Transchromatic Generalized Character Maps}

\author{Nathaniel Stapleton}
\givenname{Nathaniel}
\surname{Stapleton}
\address{Massachusetts Institute of Technology,
Department of Mathematics
2-233
77 Massachusetts Avenue
Cambridge, MA 02139-4307}
\email{nstapleton@math.mit.edu}
\urladdr{}


\arxivreference{}
\arxivpassword{}

\volumenumber{}
\issuenumber{}
\publicationyear{}
\papernumber{}
\startpage{}
\endpage{}
\doi{}
\MR{}
\Zbl{}
\received{}
\revised{}
\accepted{}
\published{}
\publishedonline{}
\proposed{}
\seconded{}
\corresponding{}
\editor{}
\version{}

\newtheorem{thm}{Theorem}[section]    
\newtheorem{prop}[subsection]{Proposition}
\newtheorem{example}[subsection]{Example}

\newtheorem*{mainthm}{Theorem}
\newtheorem{lemma}[subsection]{Lemma}          
\newtheorem{cor}[subsection]{Corollary}          
\theoremstyle{definition}
\newcommand{\powser}[1]{[\![#1]\!]}
\newcommand{\pdiv}{$p$-divisible }
\newcommand{\G}{\mathbb{G}}

\newcommand{\Sect}{\mathcal{O}}
\newcommand{\QZ}{\Q_p/\Z_p}  
\newcommand{\Zp}[1]{\Z/p^{#1}}

\newcommand{\al}{\alpha}
\newcommand{\Lk}{\Lambda_k}
\newcommand{\Lj}{\Lambda_j}
\newcommand{\lra}[1]{\overset{#1}{\longrightarrow}}
\newcommand{\lla}[1]{\overset{#1}{\longleftarrow}}
\newcommand{\Prod}[1]{\underset{#1}{\prod}}
\newcommand{\Coprod}[1]{\underset{#1}{\coprod}}
\newcommand{\Colim}[1]{\underset{#1}{\colim}}
\newcommand{\Lim}[1]{\underset{#1}{\lim}}
\newcommand{\E}{E_{n}}
\newcommand{\LE}{L_{K(t)}E_n}
\newcommand{\De}{\Delta}

\newcommand{\p}{\mathcal{P}}
\newcommand{\Lt}{L_t}
\newcommand{\mt}{I_t}
\newcommand{\Fix}{\text{Fix}_{n-t}}
\DeclareMathOperator{\Aut}{Aut}
\DeclareMathOperator{\im}{im}
\DeclareMathOperator{\Hom}{Hom}
\DeclareMathOperator{\colim}{colim}

\DeclareMathOperator{\Iso}{Iso}
\DeclareMathOperator{\Spec}{Spec}
\DeclareMathOperator{\Spf}{Spf}
\DeclareMathOperator{\Mor}{Mor}
\DeclareMathOperator{\Id}{Id}
\DeclareMathOperator{\Tot}{Tot}
\DeclareMathOperator{\Top}{Top}
\DeclareMathOperator{\Spectra}{Spectra}

\DeclareMathOperator{\Fixit}{Fix}
\def\co{\colon\thinspace}

\begin{document}

\begin{abstract}    

The generalized character map of Hopkins, Kuhn, and Ravenel \cite{hkr} can be interpreted as a map of cohomology theories beginning with a height n cohomology theory $E$ and landing in a height $0$ cohomology theory with a rational algebra of coefficients that is constructed out of $E$. We use the language of \pdiv groups to construct extensions of the generalized character map for Morava $E$-theory to every height between $0$ and $n$.

\end{abstract}

\maketitle


\section{Introduction}
In \cite{hkr}, Hopkins, Kuhn, and Ravenel develop a way to study cohomology rings of the form $E^*(EG\times_G X)$ in terms of a character map. The map developed was later used by Ando in \cite{Isogenies} to study power operations for Morava $\E$ and by Rezk in \cite{logarithmic} to construct the logarithmic cohomology operation for $\E$. Hopkins, Kuhn, and Ravenel's character map can be interpreted as a map of cohomology theories beginning with a height $n$ cohomology theory $E$ and landing in a height $0$ cohomology theory with a rational algebra of coefficients that they construct out of $E$. In this paper we use the language of \pdiv groups to extend their construction so that the character map can land in every height $0 \leq t < n$. 

We provide motivation and summarize the main result. Let $K$ be complex $K$-theory and let $R(G)$ be the complex representation ring of a finite group $G$. Consider a complex representation of $G$ as a $G$-vector bundle over a point. Then there is a natural map $R(G) \rightarrow K^0(BG)$. This takes a virtual representation to a virtual vector bundle over $BG$ by applying the Borel construction $EG \times_G -$. It is a special case of the Atiyah-Segal completion theorem \cite{atiyahcharacters} from the $60$'s that this map is an isomorphism after completing $R(G)$ with respect to the ideal of virtual bundles of dimension $0$. 

Let $L$ be a minimal characteristic zero field containing all roots of unity, and let $Cl(G;L)$ be the ring of class functions on $G$ taking values in $L$. A classical result in representation theory states that $L$ is the smallest field such that the character map
\[
R(G) \lra{} Cl(G,L)
\]
taking a virtual representation to the sum of its characters induces an isomorphism $L \otimes R(G) \lra{\cong} Cl(G;L)$ for every finite $G$.

Let $\E$ be the Morava $\E$-theory associated to the universal deformation of a height $n$ formal group law over a perfect field $k$ of characteristic $p$. Hopkins, Kuhn, and Ravenel build, for each Morava $E$-theory $\E$, an equivariant cohomology theory that mimics the properties of $Cl(G,L)$ and is the receptacle for a map from Borel equivariant $\E$. They construct a flat even periodic $\E^*$-algebra $L(\E)^*$ and define, for $X$ a finite $G$-CW complex, the finite $G$-CW complex 
\[
\Fixit_n(X) = \Coprod{\al \in \hom(\Z_{p}^{n},G)} X^{\im \al}. 
\]
Then they define a Borel equivariant cohomology theory
\[
L(\E)^*(\Fixit_n(X))^G = \big(L(\E^*)\otimes_{\E^*} \E^*(\Fixit_n(X))\big)^G
\]
and construct a map of Borel equivariant cohomology theories 
\[
\E^*(EG\times_G X) \lra{} L(\E)^*(\Fixit_n(X))^G.
\]
The codomain of this map is closely related to the class functions on $G$ taking values in $L(\E)^*$. In fact, when $X$ is a point, the codomain reduces to ``generalized class functions'' on 
\[
\hom(\Z_{p}^n,G) = \{(g_1,\ldots,g_n)|g_{i}^{p^h} = e \text{ for some } h, [g_i,g_j] = e\}
\]
considered as a $G$-set by pointwise conjugation. 
As in the case of the representation ring of a finite group, there is an isomorphism
\[
L(\E)^*\otimes_{\E^*}\E^*(EG\times_G X) \lra{\cong} L(\E)^*(\Fixit_n(X))^G.
\]

Let $\G_{\E}$ be the formal group associated to $\E$ and $\G_{\E}[p^k]$ the subscheme of $p^k$-torsion. The ring $L(\E)^0$ satisfies an important universal property: it is the initial ring extension of $p^{-1}\E^0$ such that for all $k$, $\G_{\E}[p^k]$, when pulled back over $L(\E)^0$, is canonically isomorphic to the constant group scheme $(\Zp{k})^n$. 
\[
\xymatrix{(\Zp{k})^n \ar[r] \ar[d] & \G_{\E}[p^k] \ar[d] \\
			\Spec(L(\E)^0) \ar[r] & \Spec(\E^0)}
\]

In this paper we will take advantage of the fact that this result can be rephrased in the language of \pdiv groups. Let $R$ be a ring. A \pdiv group over $R$ of height $n$ is an inductive system $(G_v, i_v)$ such that
\begin{enumerate}
\item $G_v$ is a finite free commutative group scheme over $R$ of order $p^{vn}$.
\item For each $v$, there is an exact sequence 
\[
0 \lra{} G_v \lra{i_v} G_{v+1} \lra{p^v} G_{v+1}
\]
where $p^v$ is multiplication by $p^v$ in $G_{v+1}$.
\end{enumerate}

Associated to every formal group $\G$ over a $p$-complete ring $R$ is a \pdiv group
\[
\G \rightsquigarrow \G[p] \lra{i_1} \G[p^2] \lra{i_2} \ldots.
\]
This is the ind-group scheme built out of the $p^k$-torsion for varying $k$. The only constant \pdiv groups are products of $\QZ$. The ring that Hopkins, Kuhn, and Ravenel construct is the initial extension of $p^{-1}\E^0$ such that the \pdiv group associated to $\G_{\E}$ pulls back to a constant \pdiv group.

For $\G_{\E}$, we have $\Sect_{\G_{\E}[p^k]} \cong \E^0(B\Z/p^k) = \pi_{0}F(B\Z/p^k,\E)$, the homotopy groups of the function spectrum. The pullback of $\G_{\E}[p^k]$ constructed by Hopkins, Kuhn, and Ravenel in \cite{hkr} factors through $\pi_{0}L_{K(0)}(F(B\Z/p^k,\E))$, the rationalization of the function spectrum. The spectrum of this Hopf algebra is the $p^k$-torsion of an \'etale \pdiv group. Rezk noted that there are higher analogues of this: Fix an integer $t$ such that $0\leq t<n$. Then $\Spec$ of $\pi_{0}(L_{K(t)}F(B\Z/p^k,\E))$ gives the $p^k$-torsion of a \pdiv group $\G$ over $L_{K(t)}\E^0$.

We prove that the standard connected-\'etale short exact sequence for \pdiv groups holds for $\G$ over $\LE$. There is a short exact sequence
\[
0 \lra{} \G_{0} \lra{} \G \lra{} \G_{et} \lra{} 0,
\]
where $\G_0$ is the formal group associated to $L_{K(t)}\E$ and $\G_{et}$ is an ind-\'etale \pdiv group. The height of $\G$ is the height of $\G_0$ plus the height of $\G_{et}$. 

These facts suggest that there may be results similar to those of \cite{hkr} over a ring for which the \pdiv group has a formal component, but for which the \'etale part has been made constant. This is the main theorem of the paper. 

First we define, for $X$ a finite $G$-CW complex, the finite $G$-CW complex
\[
\Fix(X) = \Coprod{\al \in \hom(\Z_{p}^{n-t},G)}X^{\im \al}.
\]

\begin{mainthm} \label{main}
For each $0 \leq t < n$ there exists an $\LE^0$-algebra $C_t$ such that the pullback 
\[
\xymatrix{\G_0 \oplus \QZ^{n-t} \ar[r] \ar[d] & \G \ar[r] \ar[d] & \G_{\E} \ar[d] \\
			\Spec(C_t) \ar[r] & \Spec(L_{K(t)}\E^0) \ar[r] & \Spec{\E^0} }
\]
is the sum of a height $t$ formal group by a constant height $n-t$ \pdiv group. The ring $C_t$ is non-zero and flat over $\E^0$ and can be used to make a height $t$ cohomology theory. Define
\[
C_t^*(EG\times_G \Fix(X)) = C_t \otimes_{\LE^0}\LE^*(EG \times_G \Fix(X)).
\]

For all finite $G$ we construct a map of equivariant theories on finite $G$-CW complexes
\[
\Phi_G \co \E^*(EG\times_G X) \lra{} C_{t}^*(EG \times_G \Fix(X))
\]
such that
\[
C_t\otimes_{\E^0}\Phi_G \co C_t\otimes_{\E^0}\E^*(EG\times_G X) \lra{\cong} C_{t}^*(EG \times_G \Fix(X))
\]
is an isomorphism of equivariant cohomology theories. The map of Hopkins, Kuhn, and Ravenel is recovered when $t=0$.
\end{mainthm}
This map is intimately related to the algebraic geometry of the situation. In fact, when $X = *$ and $G = \Z/p^k$ this map recovers the global sections of the map on $p^k$-torsion $\G_0[p^k] \oplus (\Z/p^k)^{n-t} \lra{} \G_{\E}[p^k]$.

\paragraph{Acknowledgements} It is a pleasure to thank Charles Rezk for his help and guidance over several years. I would like to thank Matt Ando, Mark Behrens, Martin Frankland, David Gepner, Bert Guillou, David Lipsky, Haynes Miller, Jumpei Nogami, Rekha Santhanam, Olga Stroilova, and Barry Walker for helpful conversations and encouragement. I would also like to thank the referee for many helpful suggestions.

\section{Transchromatic Geometry}
Let $0\leq t<n$ and fix a prime $p$. In this section we study the \pdiv group obtained from $\G_{\E}$ by base change to $\pi_0 L_{K(t)}\E$. In the first section we prove that it is the middle term of a short exact sequence of \pdiv groups
\[
0 \lra{} \G_{0} \lra{} \G \longrightarrow \G_{et} \lra{} 0,
\]
where the first group is formal and the last is \'etale. In the second section we construct the ring extension of $\pi_0 L_{K(t)}\E$ over which the \pdiv group splits as a sum of a height $t$ formal group and a constant height $n-t$ \'etale \pdiv group.
\subsection{The Exact Sequence}
\label{exactsequence}
This paper will be concerned with the Morava $E$-theories $E_n$ and their localizations with respect to Morava $K(t)$-theory for $0\leq t<n$: $L_{K(t)}E_n$. $E_n$ is an even periodic height $n$ theory and $\LE$ is an even periodic height $t$ theory. Basic properties of these cohomology theories can be found in (\cite{Notes}, \cite{Hovey-vn}, \cite{hkr}, \cite{Nilpotence}) for instance. Let $k$ be a perfect field of characteristic $p$. The coefficients of these theories are
\begin{align*}
\E^0 &\cong W(k)\powser{u_1, \ldots , u_{n-1}} \\
\LE^0 &\cong W(k)\powser{u_1, \ldots , u_{n-1}}[u_t^{-1}]^{\wedge}_{(p, \ldots ,u_{t-1})}. 
\end{align*}
The second isomorphism follows from Theorem 1.5.4 in \cite{Hovey-vn}. Thus the ring $\LE^0$ is obtained from $\E^0$ by inverting the element $u_t$ and then completing with respect to the ideal $(p,u_1,\ldots,u_{t-1})$.

Let $E$ be one of the cohomology theories above. Classically, it is most common to study these cohomology theories in terms of the associated formal group $\G_E = \Spf(E^0(BS^1))$. However, in this paper we will be studying these cohomology theories in terms of their associated \pdiv group. First we fix a coordinate for the formal group $x\co \Sect_{\G_E} \cong E^0\powser{x}$. This provides us with a formal group law $\G_{E}(x,y) \in E^0\powser{x,y}$. The coordinate can be used to understand the associated \pdiv group. 

Let $\G_E[p^k] = \Spec(E^0(B\Zp{k})) = \hom_{E^0}(E^0(B\Zp{k}),-)$. As $B\Z/p^k$ is an H-space, $E^0(B\Z/p^k)$ is a Hopf algebra, and $\G_E[p^k]$ is a commutative group scheme. The following is a classical theorem and can be found in \cite{hkr},\cite{RW}.
\begin{thm} \label{groupcoh}
Given a generator $\beta^k \in (\Zp{k})^* = \hom(\Zp{k},S^1)$, there is an isomorphism $E^0(B\Zp{k}) \cong E^0\powser{x}/([p^k](x))$, where $[p^k](x)$ is the $p^k$-series for the formal group law associated to $E$. 
\end{thm}

The dual is needed because $\Zp{k} \lra{} S^1$ induces $E^0(BS^1) \lra{} E^0(B\Zp{k})$. This allows us to use the coordinate for the formal group in order to understand the codomain. The Weierstrass preparation theorem implies the following result.
\begin{prop} \label{wprep} (\cite{hkr}, Proposition 5.2) If the height of $E$ is $n$ then $E^{0}\powser{x}/([p^k](x))$ is a free $E^{0}$-module with basis $\{1,x, \ldots x^{p^{kn}-1}\}$.
\end{prop}
Thus we see that $\G_E[p^k]$ is a finite free group scheme of order $p^{kn}$. We now have the group schemes that we would like to use to form a \pdiv group. We must define the maps that make them into a \pdiv group.

For each $k$ fix a generator $\beta^k \in (\Zp{k})^*$. Define $i_k \co \Zp{k} \lra{} \Zp{k+1}$ to be the unique map such that $\beta^{k+1}\circ i_k = \beta^k$. This provides us with a fixed isomorphism 
\[
\Colim{k} \text{ } \Zp{k} \lra{\cong} S^1[p^\infty] \subset S^1.
\]
Then, with the coordinate,
\[
i_{k}^* = E^0(Bi_k)\co E\powser{x}/([p^{k+1}](x)) \lra{} E\powser{x}/([p^k](x))\co x \mapsto x.
\]
The spectrum of this map is the inclusion $i_k \co \G_E[p^k] \lra{} \G_E[p^{k+1}]$ and makes the inductive sequence 
\[
\G_E[p] \lra{i_1} \G_E[p^2] \lra{i_2} \ldots
\]
a \pdiv group.

Before continuing we establish some notation. Let $\Lt = \LE^0$ and $\mt = (p, u_1, \ldots u_{t-1})$. Note that $\mt$ is not necessarily a maximal ideal. For a scheme $X$ over $\Spec(R)$ and a ring map $R \lra{} S$, let 
\[
S \otimes X = \Spec(S)\times_{\Spec(R)} X.
\]
Given a \pdiv group $\G_E$ over $E^0$ and a ring map $E^0 \lra{} S$, let $S\otimes \G_E$ be the \pdiv group such that $(S\otimes \G_E)[p^k] = S\otimes (\G_E[p^k])$. 

Here we collect a few facts (\cite{Notes}) regarding the $p^k$-series for the formal group law $\G_{\E}(x,y)$ that we will need later. For $0 \leq h < n$,
\[
[p^k](x) = [p^k]_h(x^{p^{kh}}) = (u_h)^{(p^{hk}-1)/(p^h-1)}(x^{p^{kh}}) +\ldots  \bmod \text{ } (p,u_1, \ldots u_{h-1}).
\]
In particular,
\[
[p](x) = [p]_h(x^{p^h}) = u_hx^{p^h} + \ldots  \bmod \text{ } (p,u_1, \ldots u_{h-1}).
\]
There is a localization map $\E \lra{} \LE$ that induces $\E^0 \lra{} \Lt$, and $\G_{\LE}(x,y)$ is obtained from $\G_{\E}(x,y)$ by applying this map to the coefficients of the formal group law. The Weierstrass preparation theorem implies that
\[
[p^k](x) = f_k(x)w_k(x)
\]
in $\E^0\powser{x}$, where $f_k(x)$ is a monic degree $p^{kn}$ polynomial and $w_k(x)$ is a unit. In $\Lt\powser{x}$,
\[
[p^k](x) = g_k(x)v_k(x),
\]
where $g_k(x)$ is a monic degree $p^{kt}$ polynomial and $v_k(x)$ is a unit. From now on the symbol $[i](x)$ will stand for the $i$-series for the formal group law $\G_{\E}(x,y)$ as an element of $\E^0\powser{x}$ or an element of $L_t\powser{x}$. The above discussion implies that $[i](x)$ may have different properties depending on the ground ring. The ground ring should be clear from context. 

Now we focus our attention on the \pdiv group $\Lt\otimes \G_{\E}$.

\begin{prop}
The \pdiv group $\Lt\otimes\G_{\E}$ has height $n$ and the \pdiv group $\G_{\LE}$ has height $t$. The \pdiv group $\G_{\LE}$ is a sub-\pdiv group of $\Lt \otimes \G_{\E}$.
\end{prop}
\begin{proof}
The heights of the \pdiv groups follow immediately from Prop \ref{wprep}.
We show that $\G_{\LE}$ is a sub-\pdiv group of $\Lt \otimes \G_{\E}$ on the level of $p^k$-torsion. We have the following sequence of isomorphisms:
\begin{align*}
\Lt\otimes_{\E^0}\E^0(B\Z/p^k) &\cong \Lt\otimes_{\E^0}\E^0\powser{x}/([p^k](x)) \\
 &\cong \Lt\otimes_{\E^0} \E^0[x]/(f_k(x)) \\ 
&\cong \Lt[x]/(f_k(x))
\end{align*}
and
\begin{align*}
(\LE)^0(B\Z/p^k) &\cong \Lt\powser{x}/([p^k](x)) \\
 &\cong \Lt[x]/(g_k(x)), 
\end{align*}
where $f_k(x)$ and $g_k(x)$ are as above.
The canonical map 
\[
\E^0(B\Zp{k}) \lra{} (L_{K(t)}\E)^0(B\Zp{k})
\]
implies that $f_k(x) = g_k(x)h_k(x)$ as polynomials where $h_k(x) = v_k(x)/w_k(x)$. This implies that $\G_{\LE}[p^k]$ is a subgroup scheme of $\Lt \otimes \G_{\E}[p^k]$ and it is clear that the structure maps fit together to give a map of \pdiv groups.
\end{proof}

Throughout this paper we use very little formal algebraic geometry. However, it is important to note a certain fact that only holds over
\[
\Spf_{\mt}(\Lt) = \Colim{} \big(\Spec(\Lt/\mt) \lra{} \Spec(\Lt/\mt^2) \lra{} \Spec(\Lt/\mt^3) \lra{} \ldots \big).
\]
\begin{prop}
Over $\Spf_{\mt}(\Lt)$, $\G_{\LE}$ is the connected component of the identity of $\Lt \otimes \G_{\E}$.
\end{prop}
\begin{proof}
This is the same as saying that $\G_{\LE}$ is the formal component of the \pdiv group $\Lt \otimes \G_{\E}$ over $\Spf_{\mt}(\Lt)$. Once again we prove this by working with the $p^k$-torsion. The proof has two steps. First we give an explicit decomposition of $\Sect_{\Lt \otimes \G_{\E}[p^k]}$ as a product of two rings. We identify one of the factors as $\Sect_{\G_{\LE}[p^k]}$. Secondly, we show that $\Sect_{\G_{\LE}[p^k]}$ is connected.

The rings $\Lt[x]/(f_k(x))$, $\Lt[x]/(g_k(x))$, and $\Lt[x]/(h_k(x))$ are all finitely generated free $\Lt$-modules because the polynomials are monic. Thus the natural map
\[
\Lt[x]/(f_k(x)) \lra{} \Lt[x]/(g_k(x))\times \Lt[x]/(h_k(x))
\]
has the correct rank on both sides. We must show that it is surjective.

By Nakayama's lemma (\cite{Eisenbud}, Corollary 4.8), it suffices to prove this modulo $\mt$. Modulo $\mt$, $g_k(x) = x^{p^{kt}}$ and $h_k(x)$ has constant term a unit, a power of $u_t$, and smallest nonconstant term degree $x^{p^{kt}}$. Thus $(g_k(x))$ and $(h_k(x))$ are coprime and the map is an isomorphism. Note that 
\[
\Sect_{\G_{\LE}[p^k]} \cong \Lt[x]/(g_k(x)).
\]

To prove the connectedness, first note that $\Spf_{\mt}(\Lt)$ is connected. This is because the underlying space of $\Spf_{\mt}(\Lt)$ is the underlying space of
\[
\Spec(\Lt/\mt) = \Spec(k\powser{u_{t},\ldots, u_{n-1}}[u_{t}^{-1}]),
\]
which has no zero divisors.

For the same reason, it suffices to check that 
\[
\Spec(\Sect_{\G_{\LE}[p^k]}/\mt)
\]
is connected to prove that $\G_{\LE}[p^k]$ is connected over $\Spf_{\mt}(\Lt)$. However, we have the following sequence of isomorphisms
\begin{align*}
\Sect_{\G_{\LE}[p^k]}/\mt &\cong (\Lt/\mt)\powser{x}/([p^k](x)) \\
 &\cong (\Lt/\mt)\powser{x}/(x^{p^k}) \\
 &\cong (\Lt/\mt)[x]/(x^{p^k}).
\end{align*}
The last ring is connected because, modulo nilpotents, it is just $\Lt/\mt$.
\end{proof}

We conclude that the connected component of the identity of $\Lt\otimes\G_{E_n}[p^k]$ is isomorphic to $\G_{\LE}[p^k]$.

Let $\G = \Lt\otimes\G_{E}$ and $\G_0 = \G_{\LE}$.

Recall that we are working to prove that the \pdiv group $\G$ is the middle term of a short exact sequence
\[
0 \lra{} \G_{0} \lra{} \G \longrightarrow \G_{et} \lra{} 0,
\]
where the first \pdiv group is formal and the last is \'etale. This will come from an exact sequence at each level
\[
0 \lra{} \G_{0}[p^k] \lra{} \G[p^k] \lra{} \G_{et}[p^k] \lra{} 0.
\] 

Next we show that $\G_{et}[p^k]$ is in fact \'etale (as its nomenclature suggests). We begin by giving a description of the global sections of $\G_{et}[p^k]$.

The group scheme $\G_{et}[p^k]$ is the quotient of $\G[p^k]$ by $\G_{0}[p^k]$. It can be described as the coequalizer of
\[
\xymatrix{\G_0[p^k] \times \G[p^k] \ar@<1ex>[r]^(.6){\mu} \ar[r]_(.6){\pi} & \G[p^k]},
\]
where the two maps are the multiplication, $\mu$, and the projection, $\pi$.

Using the methods of Demazure-Gabriel in \cite{groupes} as explained in Section 5 of \cite{subgroups}, we can describe the global sections of $\G_{et}[p^k]$, or the equalizer
\[
\Sect_{\G_{et}[p^k]} \lra{} \Sect_{\G[p^k]} \rightrightarrows \Sect_{\G[p^k]}\otimes\Sect_{\G_0[p^k]}
\]
by using a norm construction.

Let $R \lra{f} S$ be a map of rings, where $S$ is a finitely generated free $R$-module. Given $u \in S$, multiplication by $u$ is an $R$-linear endomorphism of $S$. Thus its determinant is an element of $R$. Let $N_f \co S \lra{} R$ be the multiplicative norm map
\[
N_{f}(u) = \det(-\times u), 
\]
the map that sends $u \in S$ to the determinant of multiplication by $u$. $N_{f}$ is not additive. The basic properties of the norm are described in \cite{subgroups}.

It is shown in \cite{subgroups} that for $x \in \Sect_{\G[p^k]}$, $N_{\pi}\mu^*(x)$, which is naturally an element of $\Sect_{\G[p^k]}$, actually lies in $\Sect_{\G_{et}[p^k]}$. Let $y = N_{\pi}\mu^*(x)$. It is also shown that if $i\co\G_0[p^k] \lra{} \G[p^k]$ is the inclusion, then $i^*y = 0$.

Using these facts we arrive at the following proposition.

\begin{prop} \label{jk}
There is an isomorphism $\Sect_{\G_{et}[p^k]} \cong \Lt[y]/(j_k(y))$ where $j_k(y)$ is a monic polynomial of degree $p^{k(n-t)}$.
\end{prop}
\begin{proof}
Recall that we have given more explicit descriptions of $\Sect_{\G[p^k]}$ and $\Sect_{\G_0[p^k]}$:
\begin{align*}
\Sect_{\G[p^k]} &\cong \Lt[x]/(f_k(x)) \\
\Sect_{\G_0[p^k]} &\cong \Lt[x]/(g_k(x)).
\end{align*}
To begin we see that $i^*(y) = 0$ implies that $g_k(x) | y$ in $\Lt[x]/(f_k(x))$.

It turns out to be easy to understand $y$ $\bmod$ $I_t$. This is because the norm commutes with quotients. When working $\bmod$ $\mt$, $g_k(x) = x^{p^{kt}}$. So $\Sect_{\G[p^k]\times \G_0[p^k]}/\mt \cong (\Lt/\mt)[x,z]/(f_k(x),z^{p^{kt}})$ and $\mu^* x = x$ $\bmod$ $z$ because $\mu^* x$ is the image of the formal group law in $(\Lt/\mt)[x,z]/(f_k(x),z^{p^{kt}})$. So the matrix for multiplication by $\mu^* x$ in the basis ${1,z,\ldots,z^{p^{kt}-1}}$ is upper triangular with diagonal entries $x$. Thus $y = N_{\pi}\mu^*x = x^{p^{kt}} \bmod$ $\mt$.

The $L_t$-algebra $\Sect_{\G_{et}[p^k]}$ is a subalgebra of $\Sect_{\G[p^k]}$ that is free as an $\Lt$-module. As $y \in \Sect_{\G_{et}[p^k]}$ so is $y^l = N_{\pi}\mu^{*}{x^l}$. Now as each of $\{1,y, \ldots, y^{p^{(n-t)k}-1}\}$ are linearly independent $\bmod$ $\mt$, they are linearly independent in $\Lt[x]/(f_k(x))$. Also Nakayama's lemma implies that they are part of a basis for $\Lt[x]/(f_k(x))$, because the set is part of a basis $\bmod$ $\mt$. A quick count of this set shows that it does span $\Sect_{\G_{et}[p^k]}$. Thus $\Sect_{\G_{et}[p^k]} \cong \Lt[y]/(j_k(y))$ where $j_k(y)$ is a monic polynomial.
\end{proof}

\begin{cor}
There is an isomorphism 
\[
\Sect_{\G_{et}[p^k]}/\mt \cong (\Lt/\mt)\otimes_{\E^0} \E^0\powser{y}/([p^k]_t(y)).
\]
\end{cor}
\begin{proof}
We have noted that $[p^k](x) = f_k(x) \cdot w_k(x)$ where $w_k(x)$ is a unit, and that
\[
[p^k](x) = [p^k]_t(x^{p^{kt}}) = u_tx^{p^{tk}} + \ldots \bmod \mt.
\]
Thus $[p^k]_t(x^{p^{kt}}) = f_{k,t}(x^{p^{kt}})w_{k,t}(x^{p^{kt}}) \bmod \mt$, where $w_{k,t}$ is a unit. In the previous proposition we showed that $j_k(y) = f_{k,t}(y) \bmod \mt$. Thus
\[
\Sect_{\G_{et}[p^k]}/\mt \cong (\Lt/\mt)[y]/(j_k(y)) \cong (\Lt/\mt)\otimes_{\E^0} \E^0\powser{y}/([p^k]_t(y)).
\]
\end{proof}

\begin{prop}
The scheme $\G_{et}[p^k]$ is an \'etale scheme.
\end{prop}
\begin{proof}
Consider $\mt$ as an ideal of $\E$. We show that $u_t|[p^k]^{'}_t(y)$ in $(\E/\mt)\powser{y}/[p^k]_t(y)$. Indeed, 
\[
[p^k]_t(y) = [p^k]_{t+1}(y^{p^k}) \bmod u_t
\]
and as we are working in characteristic $p$, $[p^k]^{'}_{t}(y) = 0 \bmod u_t$.

Now it is clear that
\[
1 \otimes [p^k]_t'(y) = 1\otimes (u_t+\ldots) = u_t\otimes (1 + \ldots) 
\]
is a unit in $(\Lt/\mt)\otimes_{\E^0} \E^0\powser{y}/[p^k]_t(y)$.

But now $[p^k]_t(y) = j_{k}(y)w_{k,t}(y)$ implies that 
\[
[p^k]_t'(y) = j_k'(y)w_{k,t}(y) + j_k(y)(w_{k,t})'(y) \bmod \mt.
\]
As the second term is divisible by $j_k(y)$, it vanishes. We see that $j_k'(y)$ is a unit, and as units lift, $j_k'(y) \in \Sect_{\G_{et}[p^k]}^{\times}$. The result follows from Corollary 3.16 in \cite{Milne}. 
\end{proof}

\subsection{Splitting the Exact Sequence}
\label{splitting}
Our goal is to algebraically construct the initial extension of $\Lt$ over which the \pdiv group $\Lt\otimes\G_{\E}$ splits as the sum of the connected part and a constant \'etale part. This is similar to work of Katz-Mazur in Section 8.7 of \cite{KM}. Although we often suppress the notation, all groups in this section are considered to be constant group schemes.

Initially we want to find the $\Lt$-algebra that represents the functor
\[ 
\hom(\QZ^{n-t},\G):\co R \mapsto \hom_{\text{\pdiv}}(R\otimes \QZ^{n-t}, R \otimes \G).
\]
This was done for $t=0$ in \cite{hkr}. The construction here is analogous, but stated more algebro-geometrically. It turns out to be convenient for working with the coordinate and for reasons of variance to use the duals of groups as well as the groups themselves.

Let $\Lk = (\Z/p^k)^{n-t}$.
The following is a corollary of Theorem \ref{groupcoh}.
\begin{cor}
Given $\Lk$ and a set $\beta_{1}^{k}, \ldots, \beta_{n-t}^{k}$ of generators of $\Lk^*$ there is an isomorphism $\E^0(B\Lk) \cong \E^0\powser{x_1,\ldots,x_{n-t}}/([p^k](x_1),\ldots,[p^k](x_{n-t}))$.
\end{cor}

In this case, one uses the map to the product $\beta_{1}^{k}\times \ldots \times \beta_{n-t}^{k}\co \Lk \lra{} S^1\times \ldots \times S^1 $ to obtain the result using the fixed coordinate.

Given a sequence of epimorphisms $\Lambda_1 \lla{\rho_2} \Lambda_2 \lla{\rho_3} \ldots$, let a coherent set of generators for the dual sequence be, for each $k$, a set of generators $\{\beta^{k}_1,\ldots,\beta^{k}_{n-t}\}$ for $\Lambda_{k}^{*}$ such that $p\cdot \beta^{k+1}_{h} = \rho_{k+1}^*(\beta^{k}_{h})$. It is clear that a coherent system of generators for the dual sequence exists for any sequence of epimorphisms of the form above.

\begin{prop}
Given a coherent system of generators for the dual sequence of the above sequence of epimorphisms, the map $\E^0(B\rho_k)\co \E^0(B\Lk) \lra{} \E^0(B\Lambda_{k+1})$ is induced by $x_i \mapsto [p](x_i)$.
\end{prop}
\begin{proof}
This follows immediately from the proof of the previous corollary and the definition of a coherent system of generators.
\end{proof}

Given $\beta_{i}^{k} \co \Lk \lra{} S^1$, a generator of the dual group, and $\beta^k \co \Z/p^k \lra{} S^1$ as defined right after Prop \ref{wprep}, there exists a unique $f_i \co \Lk \lra{} \Z/p^k$ making the triangle commute. Using $\{f_i\}_{i \in \{1,\ldots,n-t\}}$ and \cite{hkr}, Lemma 5.9, this provides an isomorphism 
\[
\E^0(B\Z/p^k)^{\otimes n-t} \lra{\cong} \E^0(B\Lk).
\] 

Next consider the functor from $\Lt$-algebras to sets given by
\[
\hom(\Lk^*,\G[p^k])\co R \mapsto \hom_{gp-scheme}(R\otimes \Lk^{*},R\otimes \G[p^k])
\]

\begin{lemma}
\label{representable}
The functor $\hom(\Lk^*,\G[p^k])$ is representable by the ring $\Lt \otimes_{\E^0} \E^0(B\Lk)$. In fact, for every choice of generators of $\Lk^*$ there is an isomorphism
\[
\hom_{\Lt}(\Lt \otimes_{\E^0} \E^0(B\Lk), -) \cong \hom(\Lk^*,\G[p^k]).
\]
\end{lemma}
\begin{proof}
Let $\{\beta_{1}^{k},\ldots, \beta_{n-t}^{k}\}$ be generators of $\Lk^*$. Recall that these generators determine the isomorphism 
\[
\Lt\otimes_{\E^0} \E^0(B\Lk) \cong \Lt\otimes_{\E^0} \E^0(B\Z/p^k)^{\otimes n-t} = \Sect_{\G[p^k]}^{\otimes (n-t)}.
\]
Let $f\co \Lk^* \lra{} R\otimes\G[p^k]$, then $f^*\co R\otimes_{\Lt}\Sect_{\G[p^k]} \lra{} \Prod{\Lk^*}R$. The generators $\{\beta_{1}^{k},\ldots,\beta_{n-t}^{k}\}$ induce $n-t$ maps of $R$-algebras $R\otimes_{\Lt}\Sect_{\G[p^k]} \lra{} R$, which induces a map $R\otimes_{\Lt}\Lt\otimes_{\E^0} \E^0(B\Lk) \lra{} R$. This is determined by the map of $\Lt$-algebras
\[
\Lt\otimes_{\E^0} \E^0(B\Lk) \lra{} R.
\]
The reverse process gives the map
\[
\hom_{\Lt}(\Lt \otimes_{\E^0} \E^0(B\Lk), -) \lra{} \hom(\Lk^*,\G[p^k]).
\]
\end{proof}

Now we permanently fix a sequence of epimorphisms
\begin{align*}
\Lambda_1 \lla{\rho_2} \Lambda_2 \lla{\rho_3} \Lambda_3 \lla{} \ldots  
\end{align*}
and a coherent set of generators for the duals, $\{\beta_{i}^{k}\}_{i \in 1,\ldots,(n-t)} \in \Lk^*$. 

Let 
\[
C_t' = \Colim{k} \text{ } \Lt\otimes_{\E^0} \E^0(B\Lk),
\]
where the colimit is over the maps $\Lt\otimes_{\E^0} \E^0(B\rho_k)$.

Next we show that $C_t'$ represents the the functor from $\Lt$-algebras to sets given by
\[
\hom(\QZ^{n-t},\G)\co R \mapsto \hom_{\text{\pdiv}}(R\otimes \QZ^{n-t}, R \otimes \G).
\]

\begin{prop}
\label{bigring}
The $\Lt$-algebra $C_t'$ represents the functor $\hom(\QZ^{n-t},\G)$.
\end{prop}
\begin{proof}
First notice that 
\begin{align*}
\hom_{\Lt}(C_t',R) &\cong \hom_{\Lt}(\Colim{k} \text{ } \Lt\otimes_{\E^0} \E^0(B\Lk),R) \\
&\cong \Lim{k} \text{ } \hom_{gp-scheme}(R\otimes \Lk^*, R\otimes \G[p^k]).
\end{align*}
An element of the inverse limit is precisely a map of \pdiv groups. This is because the diagram
\[
\xymatrix{\Lambda_{k-1}^* \ar[r]^{\rho_{k}^*} \ar[d] & \Lk^* \ar[d] \\
				\G[p^{k-1}] \ar[r]^{i_k} & \G[p^k]}
\]
commutes for all $k$.
\end{proof}

\begin{cor}
Over $C_t'$ there is a canonical map of \pdiv groups $\QZ^{n-t} \lra{} \G$.
\end{cor}
\begin{proof}
This is the map of \pdiv groups corresponding to the identity map $\Id_{C_t'}$.
\end{proof}

Because there is a canonical map $\QZ^{n-t} \lra{} \G$ over $C_t'$ there is also a canonical map $\G_0 \oplus \QZ^{n-t} \lra{} \G$ using the natural inclusion $\G_0 \lra{} \G$. 

The \pdiv group $\G_0 \oplus \QZ^{n-t}$ is a \pdiv group of height $n$ with \'etale quotient the constant \pdiv group $\QZ^{n-t}$. Over $C_t'$ the map $\G_0 \oplus \QZ^{n-t} \lra{} \G$ induces a map $\QZ^{n-t} \lra{} \G_{et}$; our next goal is to find the minimal ring extension of $C_t'$ over which this map is an isomorphism. To understand this we must analyze $\G_{et}$ and prove an analogue of Proposition 6.2 in \cite{hkr}. 

We move on to analyzing $\G_{et}$ over $C_t'$, that is, we study the canonical map $\QZ^{n-t} \lra{} \G_{et}$ and determine the minimal ring extension of $C_t'$ over which it is an isomorphism. We begin with a fact about $\G_{et}$ and some facts about finite group schemes.

\begin{prop}
\label{etale-constant}
\cite{Dem} Let $L_t/\mt \lra{} K$ be a map to an algebraically closed field. Then $K \otimes \G_{et} \cong (\QZ)^{n-t}$.
\end{prop}

Prior to proving our analogue of Prop 6.2 in \cite{hkr} we need a key lemma.
\begin{lemma}
\label{subtraction}
Let $\G$ be a finite free commutative group scheme over a ring $R$ such that $\Sect_{\G} \cong R[x]/(f(x))$ where $f(x)$ is a monic polynomial such that $x|f(x)$. Then in $\Sect_{\G\times\G} \cong R[x]/(f(x))\otimes_R R[y]/(f(y))$ the two ideals $(x-y)$ and $(x-_{\G}y)$ are equal. That is $x-_{\G}y = (x-y)\cdot u$ where $u$ is a unit.
\end{lemma}

\begin{proof}
Consider the two maps, $\De \co \G \lra{} \G\times\G$ and $i\co \ker(-_{\G}) \lra{} \G\times\G$ the inclusion of the kernel of $\G\times\G \lra{-_{\G}} \G$. By considering the functor of points it is clear that both are the equalizer of 
\[
\xymatrix{\G\times\G \ar@<1ex>[r]^(.6){\pi_1} \ar[r]_(.6){\pi_2} & \G.}
\]
Thus we have the commutative triangle
\[
\xymatrix{\ker(-_{\G}) \ar[rr]^(.55){\cong} \ar[dr] & & \G \ar[dl] \\
				& \G\times\G &}
\] 
After applying global sections it suffices to find the generators of the kernels of $\De^*$ and $i^*$. It is clear that $\De^*\co R[x]/(f(x))\otimes_R R[y]/(f(y)) \lra{} R[x]/(f(x))$ must send $x \mapsto x$ and $y \mapsto x$. Thus $(x-y)$ must be in $\ker(\De^*)$ and as $\De^*$ is surjective and we have the isomorphism $R[x]/(f(x))\otimes_R R[y]/(f(y)) / (x-y) \cong R[x]/(f(x))$, the ideal $(x-y)$ must be the whole kernel. 

To understand $i^*$, we note that $\ker(-_{\G})$ is the pullback
\[
\xymatrix{\ker(-_{\G}) \ar[r] \ar[d] \ar @{} [dr] |(.3){\ulcorner} & \G\times\G \ar[d]^{-_{\G}} \\
				e \ar[r] & \G}.
\] 
Global sections gives $\Sect_{\ker(-_{\G})} \cong R \otimes_{R[x]/(f(x))} (R[x]/(f(x))\otimes_R R[y]/(f(y)))$ where $x$ is sent to $0 \in R$ and $x-_{\G}y$ in $R[x]/(f(x))\otimes_R R[y]/(f(y))$. Thus the kernel of $i^*$ is the ideal $(x-_{\G}y)$.
\end{proof}

Let $\al \in \Lk^*$. Given a homomorphism
\[
\phi \co \Lk^* \lra{} R\otimes \G_{et}[p^k],
\]
let $\phi(\al)$ be the restriction of this map to $\al$. Using the description of $\Sect_{\G_{et}[p^k]}$ from Prop \ref{jk}, the global sections of $\phi(\al)$ are the map
\[
\phi(\al)^*\co R\otimes_{\Lt}\Lt[y]/(j_k(y)) \lra{} R.
\]
Now we define a function 
\[
\phi_y\co \Lk^* \lra{} R,
\]
by $\phi_y(\al) = \phi(\al)^*(y)$. 
The following is our analogue of Prop 6.2 in \cite{hkr}.

\begin{prop}
\label{unit-iso}
Let $R$ be an $\Lt$-algebra. The following conditions on a homomorphism
\[
\phi \co \Lk^* \lra{} R\otimes \G_{et}[p^k]
\]
are equivalent: \\
i. For all $\al \neq 0 \in \Lk^*$, $\phi_y(\al)$ is a unit. \\
ii. The Hopf algebra homomorphism
\[
R[y]/(j_k(y)) \cong R\otimes_{\Lt}\Sect_{G_{et}[p^k]} \lra{} \Prod{\Lk^*}R
\]
is an isomorphism.
\end{prop}
\begin{proof}
The proof of this proposition follows the proofs of Proposition 6.2 and Lemma 6.3 in \cite{hkr}. With respect to the bases consisting of the powers of $y$ and the obvious basis of the product ring corresponding to the elements of $\Lk^*$, the matrix of the Hopf algebra map is the Vandermonde matrix of the set $\phi_y(\Lk^*)$. 

Assuming i. we must show that the determinant of the Vandermonde matrix, $\De$, is a unit. As in \cite{hkr}, for elements $x,y$ of a ring $S$, we will write $x \sim y$ if $x=uy$ for $u$ a unit. As the matrix is Vandermonde, $\De \sim \Prod{\al_i \neq \al_j \in \Lk^*} (\phi_y(\al_i) - \phi_y(\al_j))$.

Using Prop \ref{subtraction} we have
\begin{align*}
\prod(\phi_y(\al_i) - \phi_y(\al_j)) &\sim \prod(\phi_y(\al_i) -_{\G_{et}} \phi_y(\al_j)) \\ &= \prod(\phi_y(\al_i - \al_j)) \\
&= \Prod{\al \neq 0} \Prod{\al_i - \al_j = \al} \phi_y(\al)\\
&= \Prod{\al \neq 0}\phi_y(\al)^{|\Lk^*|}. 
\end{align*}

In a ring a product of elements is a unit if and only if each of the elements is a unit. Thus the formulas above imply the reverse implication as well.
\end{proof}

As an aside, in \cite{hkr} it is also shown that $p$ must be inverted for $\phi$ to be an isomorphism. This is not the case in our situation. The analagous statement is that $u_t$ must be inverted, and it was already inverted in order to form $\G_{et}$.

Prop \ref{unit-iso} seems to imply that the smallest extension of $C_t'$ over which the canonical map $\QZ^{n-t} \lra{\phi} \G_{et}$ is an isomorphism is precisely the localization with respect to the multiplicatively closed subset generated by $\phi_y(\alpha)$ for all $\alpha \in \QZ^{n-t}$. This is what we prove next.

\begin{prop}
\label{Ct}
The functor from $\Lt$-algebras to sets given by
\[
\Iso_{\G_0[p^k]/}(\G_0[p^k]\oplus\Lk^*,\G[p^k]) \co R \mapsto \Iso_{\G_0[p^k]/}(R \otimes \G_0[p^k]\oplus\Lk^*,R \otimes \G[p^k])
\]
is representable by a nonzero ring $C^{k}_{t}$ with the property that the map $\Lt/\mt \lra{i} C^{k}_{t}/(\mt\cdot C^{k}_{t})$ is faithfully flat.
\end{prop}
\begin{proof}
Let $S_k$ be the multiplicative subset of $\Lt \otimes_{\E^0} \E^0(B\Lk)$ generated by $\phi_y(\Lk^*)$ for the canonical map $\phi \co \Lk^* \lra{} (\Lt\otimes_{\E^0}\E^0(B\Lk))\otimes \G_{et}[p^k]$.  Let $C^{k}_{t} = S_{k}^{-1}(\Lt\otimes_{\E^0}\E^0(B\Lk))$. For an $\Lt$-algebra $R$, a map from $C^{k}_{t}$ to $R$ is a map $\Lk^* \lra{\phi} R \otimes \G[p^k]$ such that $\phi_y(\al)$ is a unit in $R$ for all $\al\neq 0 \in \Lk^*$. Postcomposition with the map $R\otimes \G[p^k] \lra{} R\otimes \G_{et}[p^k]$ gives a map
\[
\Lk^* \lra{\phi} R \otimes \G_{et}[p^k]
\]
that is an isomorphism by Proposition \ref{unit-iso}. Then
\[
\hom_{L_t}(\Lt\otimes_{\E^0} \E^0(B\Lk),R) \cong \hom_{gp-scheme}(\Lk^*,R\otimes \G[p^k])
\]
and
\[
\hom_{L_t}(C^{k}_{t},R) \cong \Iso_{\G_0[p^k]/}(R\otimes \G_{0}[p^k]\oplus\Lk^*,R\otimes \G[p^k]),
\]
the isomorphisms under $\G_0[p^k]$. The last isomorphism is due to the $5$-lemma applied to (see \cite{oort} for embedding categories of group schemes in abelian categories)
\[
\xymatrix{0 \ar[r] & R\otimes \G_{0}[p^k] \ar[r] \ar[d]^{=} & R\otimes \G_{0}[p^k]\oplus \Lk^* \ar[r] \ar[d] & \Lk^* \ar[r] \ar[d]^{\cong} & 0 \\
0 \ar[r] & R\otimes \G_{0}[p^k] \ar[r] & R\otimes \G[p^k] \ar[r] & R\otimes \G_{et}[p^k] \ar[r] & 0.}
\]
Thus over $C^{k}_{t}$ there is a canonical isomorphism $\G_0[p^k]\oplus \Lk^* \lra{} \G[p^k]$.

It is vital that we show that $C^{k}_{t}$ is nonzero. We will do this by showing that $\Lt/\mt \lra{i} C^{k}_{t}/\mt$ is faithfully flat and thus an injection. The map $i$ is flat because $(\Lt\otimes_{\E^0} \E^0(B\Lk))/\mt$ is a finite module over $\Lt/\mt$ and localization is flat. To prove that it is faithfully flat we use the same argument found in \cite{hkr}. Consider a prime $\p \subset \Lt/\mt$. Let $\Lt/\mt \lra{\theta} K$ be a map to an algebraically closed field with kernel exactly $\p$. This can be achieved by taking the algebraic closure of the fraction field of the integral domain $(\Lt/\mt)/\p$.

Prop \ref{etale-constant} implies that $\G_{et}[p^k](K) \cong \Lk^*$, fixing an isomorphism provides a map $C^{k}_{t}/\mt \lra{\Psi} K$ that extends $\theta$. We have 
\[
\xymatrix{C^{k}_{t}/\mt \ar[r]^{\Psi} & K \\ \Lt/\mt \ar[u] \ar[ur]_{\theta} &}
\]
and $\ker(\Psi)$ is a prime ideal of $C^{k}_{t}$ that restricts to (or is a lift of) $\p$. The map $i$ is a flat map that is surjective on $\Spec$. This implies that it is faithfully flat (see \cite{matsumura_algebra} 4.D).
\end{proof}

The localization in the above proposition can be applied to both sides of $\Lt\otimes_{\E^0} \E^0(B\rho_k)$ and the map is well-defined. Thus over the colimit $C_t = \Colim{k} \text{ } C^{k}_{t}$, following Prop \ref{bigring}, there is a canonical isomorphism $C_t\otimes\G \cong C_t\otimes(\G_0\oplus\QZ^{n-t})$.

It follows that there is a canonical map 
\[
i_k \co \E^0(B\Lk) \lra{} \Lt\otimes_{\E^0}\E^0(B\Lk) \lra{} C_t.
\]
\begin{cor}
The ring $C_t$ is the initial $\Lt$-algebra equipped with an isomorphism 
\[
C_t \otimes \G \cong \G_0 \oplus \QZ.
\]
\end{cor}
\begin{proof}
Let $R$ be an $L_t$-algebra with a fixed isomorphism
\[
R\otimes \G[p^k] \cong R\otimes (\G_0[p^k] \oplus (\Zp{k})^{n-t}).
\]
By Lemma \ref{representable}, corresponding to the inclusion $R\otimes \Lk^* \lra{f} R\otimes \G[p^k]$ there is a map $\Lt\otimes_{\E^0}\E^0(B\Lk) \lra{} R$. 
Since $f$ induces an isomorphism onto the \'etale part of $R\otimes_{\Lt}\G[p^k]$, by Prop \ref{Ct} the map 
\[
\Lt\otimes_{\E^0}\E^0(B\Lk) \lra{} R
\]
factors through 
\[
\Lt\otimes_{\E^0}\E^0(B\Lk) \lra{} C_{t}^{k}.
\]
\end{proof}

\section{Transchromatic Generalized Character Maps}
We move on to defining the character map and we show that it induces an isomorphism over $C_t$. The point of the preceding discussion and the construction of $C_t$ is that we are going to use $C_t$ to construct a map of equivariant cohomology theories for every finite group $G$
\[
\Phi_G \co \E^*(EG\times_G X) \lra{} C_{t}^*(EG\times_G \Fix(X)).
\]
All cohomology theories are considered to be defined for finite $G$-spaces. The domain of $\Phi_G$ is Borel equivariant $\E$ and the codomain is Borel equivariant $C_t$ applied to $\Fix(X)$. The map is constructed in such a way that if $G \cong \Zp{k}$ the map of theories on a point is the global sections of the map on $p^k$-torsion $C_t\otimes(\G_{0}[p^k]\oplus(\Zp{k})^{n-t}) \lra{} \G[p^k]$.

The map $\Phi_G$ can be split into two parts, a topological part and an algebraic part. We will begin by describing the topological part. It is topological because it is induced by a map of topological spaces. After some preliminary discussion on the Borel construction and transport categories we will describe the map of topological spaces.

\subsection{The Topological Part}
Let $G$ be a finite group and $X$ a left $G$-space. Associated to $X$ as a topological space is a category that has objects the points of $X$ and only the identity morphisms (we remember the topology on the set of objects). We will abuse notation and use the symbol $X$ for the category and for the topological space. Including the action of $G$ on $X$ gives the transport category of $X$: $TX$. It is the category that has objects the points of $X$ and a morphism $g \co x_1 \lra{} x_2$ when $gx_1 = x_2$. This process associates to a group action on a topological space a category object in topological spaces.

Let $EG$ be the category with objects the elements of $G$ and a unique isomorphism between any two objects representing left multiplication in $G$. The geometric realization of the nerve of this groupoid is a model for the classical space $EG$, a contractible space with a free $G$-action. We will represent a morphism in $EG$ as $g_1 \lra{k} g_2$, where $kg_1 = g_2$. Note that this notation is overdetermined because $k = g_2g_{1}^{-1}$. For this reason we will sometimes just write $g_1 \lra{} g_2$ for the morphism.

The category $EG$ is monoidal with multiplication $m\co EG \times EG \lra{} EG$ using the group multiplication for objects and sending unique morphisms to unique morphisms. Explicitly: 
\[
m\co (g_1, h_1) \lra {} (g_2, h_2) \mapsto g_1h_1 \lra{} g_2h_2.
\]

The monoidal structure induces left and right $G$ actions on the category $EG$. Let $g_1 \lra{} g_2$ be a morphism in $EG$. Then for $g \in G$, the action is given by $g \cdot (g_1 \lra{} g_2) = gg_1 \lra{} gg_2$ and $(g_1 \lra{} g_2) \cdot g = (g_1g \lra{} g_2g)$.

\begin{prop}
As categories, $EG\times_G X \cong TX$ where the left $G$-action on the objects of $X$ is the $G$-action on the points of $X$. The geometric realization of the nerve of either of these categories is a model for the classical Borel construction.
\end{prop}
\begin{proof}
We view $EG\times_G X$ as a quotient of the product category (in fact a coequalizer). We have 
\[
(g_1,x) \lra{(k,id_x)} (g_2,x) = (e, g_1 x) \lra{(k,id_x)} (e, g_2x) \mapsto (g_1x \lra{k} g_2x) \in \Mor(TX)
\]
which is clearly an isomorphism.
\end{proof}

The space $EG \times_G X$ has a left action by $G$ induced by the left action of $G$ on $EG$. This action can be uniquely extended to a left action of $EG$ as a monoidal category. This leads to

\begin{prop}
There is an isomorphism of categories $EG \times_{EG} (EG \times_G X) \cong EG\times_G X$.
\end{prop}
\begin{proof}
Since $EG$ acts on $(EG \times_G X)$ through the action on $EG$ we have the isomorphism
\[
EG \times_{EG} (EG \times_G X) \cong (EG \times_{EG} EG) \times_G X.
\]
Now the proposition follows from the fact that
\[
EG \times_{EG} EG \cong EG.
\]
We provide some explicit formulas. We will view $EG \times_G X$ as $TX$. The isomorphism in the proposition takes objects $(g,x) = (e, gx)$ to $gx$. On morphisms 
\[
((g_1,x_1) \lra{(k,h)} (g_2,x_2)) = ((e, g_1x_1) \lra{(1, g_2hg^{-1}_1)} (e,g_2x_2)) \mapsto (g_1x_1 \lra{g_2hg^{-1}_{1}} g_2x_2).
\]
\end{proof}

Let $X$ be a finite $G$-space. Fix a $k \geq 0$ so that every map $\al \co \Z_p^{n-t} \rightarrow G$ factors through $\Lambda_k = (\Z /p^k)^{n-t}$. Define \[
\Fix(X) = \Coprod{\al \in \hom(\Z_p^{n-t},G)}X^{\im \al}.
\]
\begin{lemma}
The space $\Fix(X)$ is a $G$-space.
\end{lemma}
\begin{proof}
Let $x \in X^{\im(\al)}$ then for $g \in G$, $gx \in X^{\im(g\al g^{-1})}$.
\end{proof}

Consider the inclusion
\[
X^{\im \al} \hookrightarrow X.
\]
Using $\al$ we may define 
\[
E\Lk \times_{\Lk} X^{\im \al} \rightarrow EG \times_G X.
\]
As the action of $\Lk$ on $X^{\im \al}$ through $G$ is trivial, $E\Lk \times_{\Lk} X^{\im \al} \cong B\Lk \times X^{\im \al}$. This provides us with a map 
\[
v \co B\Lk \times \Fix(X) \cong \Coprod{\al \in \hom(\Z_{p}^{n-t},G)} B\Lk \times X^{\im \al} \rightarrow EG \times_G X.
\]
The space 
\[
\Coprod{\al \in \hom(\Z_p^{n-t},G)} B\Lk \times X^{\im \al}
\]
is a $G$-space with the action of $G$ induced by the action on $\Fix X$ together with the trivial action on $B\Lk$. With this action the $G$-space is $G$-homeomorphic to $B\Lk \times \Fix X$.
\begin{lemma}
With the $G$-action on $B\Lk \times \Fix X$ described above and the $G$-action on $EG \times_G X$ from the left action of $G$ on $EG$, the map $v$ is a $G$-map.
\end{lemma}
\begin{proof}
Let $l \in \Lk$. The following diagram commmutes:
\[
\xymatrix{(e \lra{l} e, x \in X^{\im \al}) \ar[r]^{v} \ar[d]^{g} & (e \lra{\al(l)} \al(l), x \in X) \ar[d]^{g} \\
(e \lra{l} e, gx \in X^{\im g\al g^{-1}}) \ar[r]^{v} & (g \lra{g\al(l)g^{-1}} g\al(l), x \in X).}
\]
\end{proof}

\begin{prop}
\label{top-map}
The map 
\[
B\Lk \times \Fix(X) \rightarrow EG \times_G X
\] 
extends to a map 
\[
EG \times_G (B\Lk \times \Fix(X)) \rightarrow EG \times_G X.
\]
\end{prop}

\begin{proof}
We will use the categorical formulation developed above. Applying the functor $EG\times_G(-)$ gives the map
\[
EG\times_G (B\Lk \times \Fix(X)) \rightarrow EG\times_G (EG \times_G X).
\]
Now the inclusion $G\hookrightarrow EG$ induces
\[
EG\times_G (EG \times_G X) \lra{} EG\times_{EG} (EG \times_G X) \simeq EG \times_G X.
\]
The composite of the two maps is the required extension. Explicitly:
\[
((g_1,e) \lra{(k,a)} (g_2,e), x \in X^{\im\al}) \mapsto (g_1 \lra{g_2\al(a)g_{1}^{-1}} g_2\al(a), x \in X).
\]
\end{proof}

We can do some explicit computations of this map that will be useful later. Let $X = \ast$ and $G$ be a finite abelian group. Then we have that $EG \times_G X$ is $BG$ and
\[
EG \times_G \Big(\Coprod{\al \in \hom(\Z_p^{n-t},G)} B\Lk \times X^{\im \al} \Big) \cong \Coprod{\al \in \hom(\Z_p^{n-t},G)} BG \times B\Lk.
\]
For a given $\al$ we can compute explicitly the map defined in Prop \ref{top-map}.

\begin{cor}
\label{addition}
Let $X = \ast$, $G$ be an abelian group, and fix a map $\al \co \Lk \lra{} G$. Define $+_{\al} \co \Lk \times G \longrightarrow G$ to be the addition in $G$ through $\al$. Then the map of the previous proposition on the $\al$ component
\[
B\Lk \times BG \lra{} BG
\]
is just $B+_{\al}$.
\end{cor}

\begin{proof}
Under the isomorphism $BG\times B\Lk \cong EG\times_G B\Lk$
\[ 
(e,e) \lra{(g,a)} (e,e) \mapsto (e,e) \lra{(g,a)} (g, e).
\]
Prop \ref{top-map} implies that this maps to
\[
g+\al(a).
\]
\end{proof}

Next we compute the map with $X = G/H$ for $H$ an abelian subgroup of a finite group $G$. These computations will be used in our discussion of complex oriented descent.

When the notation $\Fix(X)$ may be unclear we will use $\Fix^G(X)$ to clarify that we are using $X$ as a $G$-space. We begin by analyzing $\Fix(G/H)$ as a $G$-set. 

\begin{prop}
\label{induction}
For $H \subseteq G$ abelian, $EG\times_G \Fix^G(G/H) \simeq EH\times_H \Fix^H(\ast)$.
\end{prop}
\begin{proof}
We will produce a map of groupoids that is essentially surjective and fully faithful. The groupoid $EG\times_G \Fix^G(G/H)$ has objects the elements of $\Fix^G(G/H)$ and morphisms coming from the action of $G$. The groupoid $EH\times_H \Fix^H(\ast)$ has objects the elements of $\Fix^H(\ast)$ and morphisms from the action of $H$.

Fix an $\al \co \Z_{p}^{n-t} \lra{} G$. For $(G/H)^{\im \al}$ to be non-empty we must have that $\im \al \subseteq g^{-1}Hg$ for some $g \in G$. Why? Let $a \in \im \al$ and assume that $gH$ is fixed by $a$, then $agH = gH$ so $g^{-1}ag \in H$. Thus for $gH$ to be fixed by all $a \in \im \al$, $\im \al$ must be contained in $g^{-1}Hg$.

Every object in $\Fix^G(G/H)$ is isomorphic to one of the form $eH$. Indeed, let $gH \in (G/H)^{\im \al}$ then $g^{-1} g H = eH \in (G/H)^{g^{-1}\im \al g}$. The only objects of the form $eH$ come from maps $\al$ such that $\im \al$ is contained in $H$. We have one connected component of the groupoid $\Fix^G(G/H)$ for every $\al \co \Z_{p}^{n-t} \lra{} H$.

Now to determine the groupoid up to equivalence it suffices to work out the automorphism group of $eH \in (G/H)^{\im \al}$. Clearly the only possibilities for $g \in G$ that fix $eH$ are the $g \in H$. It turns out that all of these fix $eH$. For if $g \in H$, $geH \in (G/H)^{g \im \al g^{-1}}$, but since $H$ is abelian this is just $(G/H)^{\im \al}$. So $\Aut(eH) \cong H$ for any $eH \in \Fix^G(G/H)$.

The equivalence is now clear. We can, for example, send $\ast \in \ast^{\im \al}$ to $eH \in (G/H)^{\im \al}$ for the same $\al$ as $\im \al \in H$. 
\end{proof}

\begin{prop}
\label{induction-diagram}
For $H \subseteq G$ abelian the following diagram commutes:
\[
\xymatrix{EH\times_H B\Lk \times \Fix^H(\ast) \ar[r] \ar[d]_{\simeq} & EH \times_H \ast \ar[d]^{\simeq} \\ 
			EG\times_G B\Lk \times \Fix^G(G/H) \ar[r]  & EG\times_G G/H }
\]
\end{prop}
\begin{proof}
We will represent a morphism in $EH\times_H B\Lk \times \Fix^H(\ast)$ as a pair $((h_1, e) \lra{(h,z)} (h_2, e), \ast \in \ast^{\im \al})$. Checking commutativity on morphisms suffices (checking on identity morphisms checks it on objects). Fix an $\al$ as above. We have the following diagram morphism-wise:
\[
\xymatrix{((h_1,e) \lra{(h,z)} (h_2,e), \ast \in \ast^{\im \al}) \ar[r] \ar[d] & (h_1 \lra{h_2 \al(z) h_{1}^{-1}} h_2 \al(z), \ast) \ar[d] \\
((h_1,e) \lra{(h,z)} (h_2,e), eH \in (G/H)^{\im \al}) \ar [r] & (h_1 \lra{h_2 \al(z) h_{1}^{-1}} h_2 \al(z), eH \in (G/H)).}
\] 
\end{proof}

The map $B\Lk \times EG \times_G \Fix(X) \simeq EG \times_G \coprod B\Lk \times X^{\im \al} \rightarrow EG \times_G X$ is the map of spaces that is used to define the first part of the character map. Applying $\E$ we get
\[
\E^*(EG\times_G X) \lra{} \E^*(B\Lk\times EG\times_G \Fix(X)).
\]

\subsection{The Algebraic Part}
The algebraic part of the character map begins with the codomain above. The description of this part of the character map is much simpler. However, we must begin with a word on gradings.

Until now we have done everything in the ungraded case. This is somewhat more familiar and it is a bit easier to think about the algebraic geometry in the ungraded situation. This turns out to be acceptable because $\E$ and $\LE$ are even periodic theories. We need two facts to continue.

\begin{prop}
Let $C_{t}^*$ be the graded ring with $C_t$ in even dimensions and the obvious multiplication. Then the fact that the ring extension $\E^0 \lra{} C_t$ is flat implies that the graded ring extension $\E^* \lra{} C_{t}^*$ is flat.
\end{prop}
\begin{proof}
There is a pushout of graded rings
\[
\xymatrix{\E^0 \ar[r] \ar[d] & C_t \ar[d] \\ 
			\E^* \ar[r] & C_{t}^*,
} 
\]
where $\E^0$ and $C_t$ are taken to be trivially graded. As flatness is preserved under pushouts the proposition follows.
\end{proof}

\begin{prop}
The ring $\E^*(B\Lk)$ is even periodic.
\end{prop}
\begin{proof}
This is because $\E^*(B\Lk)$ is a free $\E^*$-module with generators in degree $0$ \cite{hkr}. Even more, the function spectrum $\E^{B\Lk}$ is a free $\E$-module as a spectrum.
\end{proof}
This is necessary to know because we will lift the map $\E^0(B\Lk) \lra{} C_t$ to a map of graded rings $\E^{*}(B\Lk) \lra{} C_{t}^*$. From now on we will use the notation $C_t$ for the ungraded ring, $C_t^*$ for the graded ring, and $C_t^*(-)$ for the cohomology theory defined by
\[
C_t^*(X) = C_t\otimes_{L_t}\LE^*(X) \cong C_t^*\otimes_{\LE^*}\LE^*(X).
\]

We return to the character map. A K\"unneth isomorphism available in this situation (\cite{hkr}, Corollary 5.10) gives
\[
\E^*(B\Lk\times EG\times_G \Fix(X)) \cong \E^*(B\Lk)\otimes_{\E^*} \E^*(EG\times_G \Fix(X))
\]
From Section \ref{splitting}, we have the maps
\[
i_k \co \E^*(B\Lk) \lra{} \Lt \otimes_{\E^0} \E^*(B\Lk) \lra{}  C_{t}^*.
\]
Also there is a map of cohomology theories $\E^*(-) \lra{} C_{t}^*(-)$ coming from the canonical map $\E \lra{} \LE$ and base extension and using the flatness of $C_t$ over $L_t$. Together these induce
\[
\E^*(B\Lk)\otimes_{\E^*} \E^*(EG\times_G \Fix(X)) \lra{} C_{t}^*(EG\times_G \Fix(X)).
\]
Precomposing with the topological map from the previous section we get the character map:
\[
\Phi_G \co \E^*(EG\times_G X) \lra{} C_{t}^*(EG\times_G \Fix(X)).
\]

It is a result of Kuhn's in \cite{Kuhn} that the codomain is in fact an equivariant cohomology theory on finite $G$-spaces. Several things must be proved to verify the claims in the statement of Theorem \ref{main}.

Recall that $\Lk$ is defined so that all maps $\Z_{p}^{n-t} \lra{}G$ factor through $\Lk$. First we show that this map does not depend on $k$.

\begin{prop} \label{independence}
The character map does not depend on the choice of $k$ in $\Lk$.
\end{prop}
\begin{proof}
Let $j > k$ and let $s = \rho_{k+1} \circ \ldots \circ \rho_j$ where $\rho_i$ is the fixed epimorphism from Section \ref{splitting}. We will use this to produce an isomorphism between the character maps depending on $j$ and $k$. Precomposition with $s$ provides an isomorphism $\hom(\Lk,G) \cong \hom(\Lj,G)$. We can use $s$ to create a homeomorphism 
\[
EG\times_G \Big( \Coprod{\al \in \hom(\Lk,G)}X^{\im \al} \Big) \cong EG\times_G \Big( \Coprod{\al \in \hom(\Lj,G)}X^{\im \al}\Big).
\]
We will abuse notation and refer to each of these spaces as $\Fix(X)$ and the map between them as the identity map $\Id$. We begin by noting that the following two diagrams commute.
\[
\xymatrix{ & B\Lk \times EG \times_G \Fix(X) \ar[dl] &  \E^*(B\Lk) \ar[dd]_{\E^*(Bs)} \ar[dr]^{i_k} & \\ 
			EG\times_G X & & & C_t \\
			& B\Lj \times EG\times_G \Fix(X) \ar[uu]_{Bs\times \Id} \ar[ul] & \E^*(B\Lj) \ar[ur]_{i_j} & } 
\]
where the diagonal arrows in the left hand diagram come from the topological part of the character map and the diagonal arrows in the right hand diagram come from the definition of $C_t$. The right hand diagram commutes by definition.

Putting these diagrams together gives the commutative diagram
\[
\xymatrix{& \E^*(B\Lk) \otimes_{\E^*} \E^*(EG\times_G \Fix(X)) \ar[dr] \ar[dd] & \\
			\E^*(EG\times_G X) \ar[ur] \ar[dr] & & C_{t}^*(EG\times_G \Fix(X)) \\
			& \E^*(B\Lj) \otimes_{\E^*} \E^*(EG\times_G \Fix(X)) \ar[ur] &}
\]
that shows the map is independent of $k$.
\end{proof}

\begin{prop} 
For $G \cong \Zp{k}$ and $X = \ast$, the codomain of the character map is the global sections of $C_t\otimes \G[p^k] \cong C_t\otimes(\G_0[p^k] \oplus \Lk^*)$. 
\end{prop}
\begin{proof}
Let $G \cong \Zp{k}$ and $X = \ast$, as $G$ is abelian it acts on $\Fix(X)$ component-wise. As $X = \ast$, 
\begin{align*}
EG \times_G \Fix(X) &= EG \times_G \Big(\Coprod{\al \in \hom(\Z_p^{n-t},G)} \ast^{\im \al}\Big) \\
							&\cong \Coprod{\hom(\Z_{p}^{n-t},G)} BG.
\end{align*}
Applying cohomology and using $\beta^k \in (\Zp{k})^* = G^*$ to identify $\Hom(\Z_{p}^{n-t},G)$ and $\Lk^*$ gives
\begin{align*}
C_{t}^0\Big(\Coprod{\hom(\Z_{p}^{n-t},G)} BG\Big) &\cong \Prod{\hom(\Z_{p}^{n-t},G)}C_{t}^0(BG) \\
								&\cong \Prod{\Lk^*}C_{t}^0(BG).
\end{align*}
The spectrum of this ring is precisely $\G_0[p^{k}]\oplus\Lk^*$.
\end{proof}

The next step is to compute the character map on cyclic $p$-groups. We begin by giving an explicit description, with the coordinate, of the global sections of the canonical map $C_t\otimes(\G_0[p^k]\oplus \Lk^*) \lra{} \G_{\E}[p^k]$. We describe the map from each summand of the domain separately. 

The global sections of the map $C_t \otimes \G_0[p^k] \lra{} \G_{\E}[p^k]$ are clearly given by 
\[
\E^0\powser{x}/([p^k](x)) \lra{x \mapsto x} C_t\powser{x}/([p^k](x)).
\]

Next we analyze the other summand. Consider the canonical map $\psi \co \Lk^* \lra{} \G_{\E}[p^k]$, which is the composition
\[
\Lk^* \lra{} C_{t} \otimes \G_{\E}[p^k] \lra{} \G_{\E}[p^k].
\]
The global sections of $\psi$ are easy to describe in terms of the coordinate. Recall that 
\[
\Lt \otimes_{\E^0} \E^0(B\Lk) \cong \Lt \otimes_{\E^0}E_{n}^0\powser{x_1,\ldots,x_{n-t}}/([p^k](x_1),\ldots,[p^k](x_{n-t}))
\]
and that there is a canonical map
\[
i_k\co \Lt \otimes_{\E^0} \E^0(B\Lk) \lra{} C_t.
\]
Using the notation of Section \ref{splitting}, let $l = c_1\cdot\beta_{1}^{k} + \ldots + c_{n-t}\cdot \beta_{n-t}^{k} \in \Lk^*$. The global sections of the restriction of $\psi$ to $l$ are the map
\[
\E^0\powser{x}/([p^k](x)) \lra{} C_t,
\]
which sends 
\[
x \mapsto i_k([c_1](x_1) +_{\G_{\E}} \ldots +_{\G_{\E}} [c_{n-t}](x_{n-t})).
\]
In order to ease the notation, define
\[
\psi_x(l) = i_k([c_1](x_1) +_{\G_{\E}} \ldots +_{\G_{\E}} [c_{n-t}](x_{n-t})).
\]
Putting these maps together for all $l \in \Lk^*$ gives
\[
E_{n}^0\powser{x}/([p^k](x)) \lra{} C_t\powser{x}/([p^k](x))\otimes_{C_t} \big( \Prod{\Lk^*}C_{t}\big) \cong \Prod{\Lk^*} C_t\powser{x}/([p^k](x)),
\]
which maps
\[
x \mapsto x +_{\G_{\E}} (\psi_x(l))_{l \in \Lk^*} \mapsto (x+_{\G_{\E}}\psi_x(l))_{l \in \Lk^*}.
\]

\begin{prop}
\label{abmap}
For $G \cong \Zp{k}$ and $X = \ast$, the character map is the global sections of $\G_0[p^k]\oplus\Lk^* \lra{} \G_{E_n}[p^k]$ described above.
\end{prop}
\begin{proof}
Fix an $\al \co \Lk \lra{} G$. Postcomposing with our fixed generator of $G^* = (\Zp{k})^*$ we get an element $c_1\cdot\beta_{1}^{k} + \ldots + c_{n-t}\cdot \beta_{n-t}^{k} \in \Lk^*$. By Prop \ref{addition} the topological part of the character map is induced by $B(-)$ of the addition map $\Lk\times G \lra{+_{\al}} G$. Applying $\E^0$ to $B+_{\al}$ gives
\[
E_{n}^0\powser{x}/([p^k](x)) \lra{} E_{n}^0\powser{x_1,\ldots,x_{n-t}}/([p^k](x_1),\ldots,[p^k](x_{n-t})) \otimes_{\E^0} E_{n}^0\powser{x}/([p^k](x)),
\]
the map sending 
\[
x \mapsto [c_1](x_1) +_{\G_{\E}} \ldots +_{\G_{\E}} [c_{n-t}](x_{n-t}) +_{\G_{\E}} x.
\]
This maps via the algebraic part of the character map,
\[
E_{n}^0\powser{x_1,\ldots,x_{n-t}}/([p^k](x_1),\ldots,[p^k](x_{n-t})) \otimes_{\E^0} E_{n}^0\powser{x}/([p^k](x)) \lra{} C_t\powser{x}/([p^k](x)),
\]
to $(x+_{\G_{\E}}\psi_x(\al))$, where $\psi_x$ is as above. Putting these together for all $\al$ gives a map
\[
E_{n}^0\powser{x}/([p^k](x)) \lra{} \Prod{\Lk^*}C_t\powser{x}/([p^k](x)).
\]
This is precisely the map constructed just prior to the proposition.
\end{proof}

\subsection{The Isomorphism}
We prove that the map of cohomology theories defined above
\[
\Phi_G \co \E^*(EG\times_G X) \lra{} C_{t}^*(EG\times_G \Fix(X)).
\]
is in fact an isomorphism when the domain is tensored up to $C_t$. We follow the steps outlined in \cite{hkr} with some added complications. 

Given a finite $G$-CW complex $X$, let $G \hookrightarrow U(m)$ be a faithful complex representation of $G$. Let $T$ be a maximal torus in $U(m)$. Then $F = U(m)/T$ is a finite $G$-space with abelian stabilizers. This means that it has fixed points for every abelian subgroup of $G$ but no fixed points for non-abelian subgroups of $G$. 

The isomorphism is proved by reducing to the case when $X$ is a point and $G$ is abelian. We first show that the cohomology of $X$ is determined by the cohomology of the spaces $X\times F^{\times h}$. Thus we can reduce to proving the isomorphism for spaces with abelian stabilizers. Using Mayer-Vietoris for the cohomology theories we can then reduce to spaces of the form $G/H \times D^l \simeq G/H$ where $H$ is abelian. Then the fact that 
\[
EG\times_G (G/H) \simeq BH
\]
and Prop \ref{induction} imply that we only need to check the isomorphism on finite abelian groups. This will follow from our previous work. 

We begin by proving the descent property for finite $G$-CW complexes. Thus we assume that the map is an isomorphism for spaces with abelian stabilizers and show that this implies it is an isomorphism for all finite $G$-spaces. 

\begin{prop}
The space $F$ is a space with abelian stabilizers.
\end{prop}
\begin{proof}
Let $A \subseteq G$ be an abelian subgroup. Then under the faithful representation above $A \subset uTu^{-1}$ for some $u \in U(m)$. Thus for $a \in A$, $a = utu^{-1}$ for some $t \in T$ and now it is clear that $A$ fixes the coset $uT$.

On the other hand, if $H \subseteq G$ is nonabelian then it cannot be contained inside a maximal torus because the representation is faithful. Therefore it will not fix any coset of the form $uT$ with $u \in U(m)$.
\end{proof}

\begin{prop}
As $F$ is a space with abelian stabilizers the realization of the simplicial space where the arrows are just the projections
\[
EF = \xymatrix{
  {\Big{|} F} &
  {F \times F} \ar@<1ex>[l] \ar@<-1ex>[l] &
  {F \times F \times F \ldots \Big{|}} \ar@<2ex>[l] \ar[l] \ar@<-2ex>[l] 
}
\]
is a space such that for $H \subseteq G$
\begin{equation*}
EF^H \simeq \left\{
\begin{array}{rl}
\emptyset & \text{if } H \text{ not abelian}\\
\ast & \text{if } H \text{ is abelian}\\
\end{array} \right.
\end{equation*}
\end{prop}
\begin{proof}
Because realization commutes with finite limits we just need to check that for $F$ a non-empty space, $EF$ is contractible. It is well known that in this situation there are backward and forward contracting homotopies to a point (see \cite{dugger_colimits}, Example 3.14).
\end{proof}

Now $EG \times_G X \simeq EG \times_G (X \times EF)$ and exchanging homotopy colimits gives

\[
EG\times_G X \simeq \xymatrix{
  {\Big{|} EG\times_G (X\times F)} &
  {EG\times_G (X \times F \times F)} \ar@<1ex>[l] \ar@<-1ex>[l] &
  {\ldots \Big{|}} \ar@<2ex>[l] \ar[l] \ar@<-2ex>[l] 
}
\]

It is important to know that $\Fix(-)$ preserves realizations.

\begin{prop}
The functor $\Fix(-)$ preserves realizations. That is, given a simplicial $G$-space $X_{\bullet}$, $\Fix(|X_{\bullet}|) \simeq |\Fix(X_{\bullet})|$.
\end{prop}
\begin{proof}
Recall that for a $G$-space $X$, $\Fix(X) = \Coprod{\al \in \hom(\Z^{n-t}_{p},G)} X^{\im \al}$.

Also recall that geometric realization as a functor from simplicial $G$-spaces to $G$-spaces is a colimit (in fact a coend), it commutes with fixed points for finite group actions (see \cite{GabrielandZisman}, Chapter 3 or \cite{may_loop_spaces}, Chapter 11), and the following diagram commutes:
\[
\xymatrix{\text{G-Spaces}^{\Delta^{op}} \ar[r]^(.55){|\text{  } |} \ar[d] & \text{G-Spaces} \ar[d]\\
          \text{Spaces}^{\Delta^{op}} \ar[r]^(.55){|\text{  }|} & \text{Spaces}}
\]
where the vertical arrows are the forgetful functor. Thus it suffices to check that $\Fix(-)$ commutes with the realization of simplicial spaces as we already know that it lands in $G$-spaces.

As colimits commute with colimits we only need to check the fixed points.
\end{proof}

We will use the Bousfield-Kan spectral sequence. For a cosimplicial spectrum $S^{\bullet}$ it is a spectral sequence 
\[
E_{2}^{s,t} = \pi^s \pi_t S^{\bullet} \Rightarrow \pi_{t-s} \Tot S^{\bullet} 
\]

As $\Sigma^{\infty}_{+} \co \Top \lra{} \Spectra$ is a left adjoint it commutes with colimits. It is also strong monoidal, taking products to smash products. The facts together imply that it preserves realizations. Let $E$ be a cohomology theory, then $\hom(|\Sigma^{\infty}_{+}X_{\bullet}|, E) \cong \Tot \hom(\Sigma^{\infty}_{+}X_{\bullet}, E)$, the totalization of the cosimplicial spectrum. The Bousfield-Kan spectral sequence begins with the homotopy of the cosimplicial spectrum $\hom(\Sigma^{\infty}_{+} X_{\bullet}, E)$ and abuts to the homotopy of $\Tot \hom(\Sigma^{\infty}_{+}X_{\bullet}, E)$.

This applies to our situation. We want to resolve 
\begin{align*}
C_{t}^*(EG\times_G \Fix(X)) &\cong \pi_{-*}\hom(\Sigma^{\infty}_{+}EG\times_G \Fix(X),C_t) \\  &\cong \pi_{-*}\hom(\Sigma^{\infty}_{+}EG\times_G \Fix(|X\times F^{\bullet}|),C_t) \\  &\cong \pi_{-*}\hom(|\Sigma^{\infty}_{+}EG\times_G \Fix(X\times F^{\bullet})|,C_t) \\  &\cong \pi_{-*}\Tot\hom(\Sigma^{\infty}_{+}EG\times_G \Fix(X\times F^{\bullet}),C_t).
\end{align*}
It follows from Prop 2.4 and 2.6 in \cite{hkr} that $\E^*(EG\times_G (X\times F^{\times h}))$ is a finitely generated free $\E^*(EG \times_G X)$-module for all $h$. From \cite{hkr} Propositions 2.4 - 2.6 we have that 
\[
\E^*(EG \times_G (X \times F \times F)) \cong \E^*(EG \times_G (X \times F) \times_{(EG \times_G X)} EG\times_G(X \times F))
\]
and this is isomorphic to
\[
\E^*(EG \times_G (X \times F)) \otimes_{\E^*(EG \times_G X)} \E^*(EG \times_G (X \times F)).
\]
We see that the cosimplicial graded $\E^*$-module
\[
\xymatrix{
  { E^{*}_{n}(EG\times_G (X\times F))} \ar@<1ex>[r] \ar@<-1ex>[r]&
  {E^{*}_{n}(EG\times_G (X \times F \times F))} \ar@<2ex>[r] \ar[r] \ar@<-2ex>[r]&
  { \ldots }  
}
\]
is in fact the Amitsur complex of the faithfully flat (even free) map $\E^*(EG \times_G X) \lra{} \E^*(EG \times_G (X \times F))$ induced by the projection. This implies that its homology is concentrated in the zeroeth degree and isomorphic to $\E^*(EG \times_G X)$. In other words the associated chain complex is exact everywhere but at the first arrow. 

The chain complex is the $E_1$ term for the Bousfield-Kan spectral sequence and we have shown that it collapses. Tensoring with $C_t$ retains this exactness as $C_t$ is flat over $\E^0$. Using our assumption regarding spaces with abelian stabilizers we now have a map of $E_1$-terms that is an isomorphism 

\[
\xymatrix{
  { C_t\otimes_{\E^0} E^{*}_{n}(EG\times_G X\times F)} \ar@<1ex>[r] \ar@<-1ex>[r] \ar[d]^{\cong}&
  {C_t\otimes_{\E^0} E^{*}_{n}(EG\times_G X \times F \times F)} \ar[d]^{\cong} \ar@<2ex>[r] \ar[r] \ar@<-2ex>[r]&
  {\ldots} \\
  {C_t^*(EG\times_G \Fix(X\times F))} \ar@<1ex>[r] \ar@<-1ex>[r] &
  {C_t^*(EG\times_G \Fix(X\times F \times F))} \ar@<2ex>[r] \ar[r] \ar@<-2ex>[r] &
  {\ldots}.
 }  
\]
As the homology of these complexes is the $E_2 = E_{\infty}$ page of the spectral sequence and the spectral sequence does converge (\cite{BK}, IX.5) to an associated graded (in this case with one term), this implies that $C_t\otimes_{\E^0} \E^*(EG\times_G X)$ and $C_{t}^*(EG\times_G \Fix(X))$ are isomorphic. Thus we have complex oriented descent.

We are reduced to proving the isomorphism for spaces with abelian stabilizers. Using an equivariant cell decomposition Mayer-Vietoris reduces this to spaces of the form $G/H\times D^n$ where $H$ is abelian and $D^n$ is the $n$-disk. Now homotopy invariance reduces this to spaces of the form $G/H$ with $H$ abelian.

\begin{prop}
The induction property holds for $G/H$ where $H \subseteq G$ is abelian. That is the following diagram commutes:
\[
\xymatrix{C_t\otimes_{\E^0} \E^*(EG\times_G G/H) \ar[r]^{C_t\otimes \Phi_G} \ar[d]^{\cong} & C_t^*(EG\times_G \Fix^G(G/H)) \ar[d]^{\cong} \\
			C_t\otimes_{\E^0} \E^*(EH\times_H *) \ar[r]^{C_t\otimes \Phi_H} & C_t^*(EH\times_H \Fix^H(*))}
\]

\end{prop}
\begin{proof}
This follows from Prop \ref{induction-diagram} and Prop \ref{independence}.
\end{proof}

We are left having to show it is an isomorphism for finite abelian groups, but we can use the K\"unneth isomorphism to reduce to cyclic $p$-groups and the isomorphism there has already been proved in Prop \ref{abmap}.

We conclude by working an example that highlights the difference between the transchromatic character maps in the case when $t>0$ and the traditional case when $t=0$.

\begin{example}
We calculate the codomain of the character map when $X = *$ and for arbitrary finite groups $G$. When $X = *$, $\Fix(*)$ is the $G$-set 
\[
\{(g_1,\ldots,g_{n-t}) | [g_i,g_j] = e, g_{i}^{p^n} = e \text{ for } n >> 0\}
\]
with action by coordinate-wise conjugation. Thus 
\[
EG\times_G\Fix(*) \simeq \Coprod{[g_1,\ldots,g_{n-t}]} BC(g_1,\ldots,g_{n-t}),
\]
the disjoint union over ``generalized conjugacy classes'' of the classifying space of the centralizer of the $(n-t)$-tuple. When $t=0$, $p$ is invertible in $C_0$ and $C_0^*(BG) \cong C_0^*$ so the codomain is class functions with values in $C_0^*$ on the set of $(n-t)$-tuples. When $t>0$ this is not the case: 
\[
C_t^*(EG\times_G\Fix(*)) \cong \Prod{[g_1,\ldots,g_{n-t}]} C_t^*(BC(g_1,\ldots,g_{n-t})).
\] 
\end{example}

%
%
%
\bibliographystyle{gtart}
\bibliography{mybib}

\begin{thebibliography}{}
\providecommand\bibmarginpar{\leavevmode\marginpar}
\def\urlstyle#1{{\tt #1}}

\bibitem{Isogenies}
\textbf{M Ando}, \href{http://dx.doi.org/10.1215/S0012-7094-95-07911-3}
  {\emph{Isogenies of formal group laws and power operations in the cohomology
  theories {$E_n$}}}, Duke Math. J. 79 (1995) 423--485

\bibitem{atiyahcharacters}
\textbf{M\,F Atiyah}, \href{http://dx.doi.org/10.1007/BF02698718}
  {\emph{{Characters and cohomology of finite groups.}}}, Publ. Math., Inst.
  Hautes Etud. Sci. 9 (1961) 247--288

\bibitem{BK}
\textbf{A\,K Bousfield}, \textbf{D\,M Kan}, \emph{Homotopy limits, completions
  and localizations}, Lecture Notes in Mathematics, Vol. 304, Springer-Verlag,
  Berlin (1972)

\bibitem{Dem}
\textbf{M Demazure}, \emph{Lectures on {p}-divisible groups}, volume 302 of
  \emph{Lecture Notes in Mathematics}, Springer-Verlag, Berlin (1986)Reprint of
  the 1972 original

\bibitem{groupes}
\textbf{M Demazure}, \textbf{P Gabriel}, \emph{Groupes alg\'ebriques. {T}ome
  {I}: {G}\'eom\'etrie alg\'ebrique, g\'en\'eralit\'es, groupes commutatifs},
  Masson \& Cie, \'Editeur, Paris (1970)Avec un appendice {{\i}t Corps de
  classes local} par Michiel Hazewinkel

\bibitem{dugger_colimits}
\textbf{D Dugger}, \href{http://pages.uoregon.edu/ddugger/hocolim.pdf} {\emph{A
  primer on homotopy colimits}}
\ Available at \setbox0\hbox{\makeatletter\@url
{http://pages.uoregon.edu/ddugger/hocolim.pdf}}
\href{http://pages.uoregon.edu/ddugger/hocolim.pdf}
{\unhbox0}

\bibitem{Eisenbud}
\textbf{D Eisenbud}, \emph{Commutative algebra}, volume 150 of \emph{Graduate
  Texts in Mathematics}, Springer-Verlag, New York (1995)With a view toward
  algebraic geometry

\bibitem{GabrielandZisman}
\textbf{P Gabriel}, \textbf{M Zisman}, \emph{Calculus of fractions and homotopy
  theory}, Ergebnisse der Mathematik und ihrer Grenzgebiete, Band 35,
  Springer-Verlag New York, Inc., New York (1967)

\bibitem{hkr}
\textbf{M\,J Hopkins}, \textbf{N\,J Kuhn}, \textbf{D\,C Ravenel},
  \emph{{Generalized group characters and complex oriented cohomology
  theories.}}, J. Am. Math. Soc. 13 (2000) 553--594

\bibitem{Hovey-vn}
\textbf{M\,A Hovey}, \href{http://dx.doi.org/10.1215/S0012-7094-97-08813-X}
  {\emph{{$v_n$}-elements in ring spectra and applications to bordism theory}},
  Duke Math. J. 88 (1997) 327--356

\bibitem{KM}
\textbf{N\,M Katz}, \textbf{B Mazur}, \emph{Arithmetic moduli of elliptic
  curves}, volume 108 of \emph{Annals of Mathematics Studies}, Princeton
  University Press, Princeton, NJ (1985)

\bibitem{Kuhn}
\textbf{N\,J Kuhn}, \emph{Character rings in algebraic topology}, from:
  ``Advances in homotopy theory ({C}ortona, 1988)'', London Math. Soc. Lecture
  Note Ser. 139, Cambridge Univ. Press, Cambridge (1989)  111--126

\bibitem{matsumura_algebra}
\textbf{H Matsumura}, \emph{Commutative algebra}, volume~56 of
  \emph{Mathematics Lecture Note Series}, second edition, Benjamin/Cummings
  Publishing Co., Inc., Reading, Mass. (1980)

\bibitem{may_loop_spaces}
\textbf{J\,P May}, \emph{The geometry of iterated loop spaces},
  Springer-Verlag, Berlin (1972)Lectures Notes in Mathematics, Vol. 271

\bibitem{Milne}
\textbf{J\,S Milne}, \emph{\'{E}tale cohomology}, volume~33 of \emph{Princeton
  Mathematical Series}, Princeton University Press, Princeton, N.J. (1980)

\bibitem{oort}
\textbf{F Oort}, \emph{Commutative group schemes}, volume~15 of \emph{Lecture
  Notes in Mathematics}, Springer-Verlag, Berlin (1966)

\bibitem{Nilpotence}
\textbf{D\,C Ravenel}, \emph{Nilpotence and periodicity in stable homotopy
  theory}, volume 128 of \emph{Annals of Mathematics Studies}, Princeton
  University Press, Princeton, NJ (1992)Appendix C by Jeff Smith

\bibitem{RW}
\textbf{D\,C Ravenel}, \textbf{W\,S Wilson},
  \href{http://dx.doi.org/10.2307/2374093} {\emph{The {M}orava {$K$}-theories
  of {E}ilenberg-{M}ac {L}ane spaces and the {C}onner-{F}loyd conjecture}},
  Amer. J. Math. 102 (1980) 691--748

\bibitem{Notes}
\textbf{C Rezk}, \emph{Notes on the {H}opkins-{M}iller theorem}, from:
  ``Homotopy theory via algebraic geometry and group representations
  ({E}vanston, {IL}, 1997)'', Contemp. Math. 220, Amer. Math. Soc., Providence,
  RI (1998)  313--366

\bibitem{logarithmic}
\textbf{C Rezk}, \href{http://dx.doi.org/10.1090/S0894-0347-06-00521-2}
  {\emph{The units of a ring spectrum and a logarithmic cohomology operation}},
  J. Amer. Math. Soc. 19 (2006) 969--1014

\bibitem{subgroups}
\textbf{N\,P Strickland},
  \href{http://dx.doi.org/10.1016/S0022-4049(96)00113-2} {\emph{Finite
  subgroups of formal groups}}, J. Pure Appl. Algebra 121 (1997) 161--208

\end{thebibliography}


\end{document}